\documentclass[a4paper,11pt,fleqn]{amsart}

\usepackage[paperwidth=216mm,paperheight=279mm,inner=3.5cm,outer=3.5cm,top=2.5cm,bottom=2.5cm]{geometry}
\usepackage[T1]{fontenc}
\usepackage[]{amsmath}
\usepackage{amsthm}
\usepackage{tikz}
\usepackage{pgfplots}
\usepackage{mathrsfs}

\theoremstyle{plain}
\newtheorem{Theorem}{Thm}[section]
\newtheorem{Thm}[Theorem]{Theorem}
\newtheorem{Lem}[Theorem]{Lemma}
\newtheorem{Cor}[Theorem]{Corollary}
\newtheorem{Prop}[Theorem]{Proposition}

\newtheorem*{Thm*}{Theorem}
\newtheorem*{Lem*}{Lemma}
\newtheorem*{Cor*}{Corollary}
\newtheorem*{Prop*}{Proposition}
\theoremstyle{definition}
\newtheorem{Def}[Theorem]{Definition}

\newtheorem{Rem}[Theorem]{Remark}

\newtheorem{Exm}[Theorem]{Example}

\newtheorem*{Def*}{Definition}
\newtheorem*{Defs*}{Definitions}
\newtheorem*{Not*}{Notation}
\theoremstyle{remark}
\newtheorem{Text}[Theorem]{}
\newtheorem*{Text'}{}

\makeatletter
\renewenvironment{proof}[1][\proofname]{\par
\pushQED{\qed}%
\normalfont \topsep6\p@\@plus6\p@\relax
\trivlist
\item[\hskip\labelsep
\scshape
#1\@addpunct{.}]\ignorespaces
}{%
\popQED\endtrivlist\@endpefalse
}
\makeatother

\newcommand\mbb{\mathbb}

\providecommand\SO{\textrm{SO}}

\providecommand\M{\textrm{M}}
\providecommand\cl{\textrm{cl}}
\renewcommand\S{\textrm{S}}
\newcommand\ol{\overline}

\newcommand\A{\mbb{A}}
\newcommand\C{\mbb{C}}
\newcommand\N{\mbb{N}}
\renewcommand\P{\mbb{P}}

\newcommand\R{\mbb{R}}
\newcommand\Z{\mbb{Z}}

\newcommand\vphi{\phi}
\renewcommand\phi{\varphi}
\renewcommand\epsilon{\varepsilon}
\renewcommand\theta{\vartheta}
\newcommand\V{\mathcal{V}}

\newcommand\im{ {\rm im}}
\newcommand\GL{ {\rm GL}}

\newcommand{\ratto}{\dashrightarrow}
\newcommand\Pol[3][n]{{#3}[{#2}_1,\ldots,{#2}_{#1}]}

\newcommand\conv{ {\rm conv}}

\begin{document}
\title{The Algebraic Boundary of $\SO(2)$-Orbitopes}
\author{Rainer Sinn}
\address{Fachbereich Mathematik und Statistik, Universit{\"a}t Konstanz, 78457 Konstanz, Germany}
\email{rainer.sinn@uni-konstanz.de}
\date{}

\begin{abstract}
Let $X\subset\A^{2r}$ be a real curve embedded into an even-dimensional affine space. In the main result of this paper, we characterise when the $r$-th secant variety to $X$ is an irreducible component of the algebraic boundary of the convex hull of the real points $X(\R)$ of $X$. This fact is then applied to $4$-dimensional $\SO(2)$-orbitopes and to the so called Barvinok-Novik orbitopes to study when they are basic closed as semi-algebraic sets. In the case of $4$-dimensional $\SO(2)$-orbitopes, we find all irreducible components of their algebraic boundary.
\end{abstract}
\maketitle

\section{Introduction}
An \emph{orbitope} is the convex hull of an orbit under a linear action of a compact real algebraic group on a real vector space. It is a compact, convex, semi algebraic set.
This paper focuses on the \emph{algebraic boundary} of $\SO(2)$-orbitopes, i.e.~the $\R$-Zariski closure of their boundary in the euclidean topology. The algebraic boundary is a central object in convex algebraic geometry. It is for example closely related to the notion of a \emph{spectrahedron} by a result of Helton and Vinnikov, see \cite{HelVinMR2292953}, section 3. The notion of a spectrahedron is of interest in convex optimization, namely semi-definite programming, see e.g.~the book by Boyd, El Ghaoui, Feron and Balakrishnan \cite{BoydMR1284712} or the article \cite{VanBoydMR1379041} by Vandenberghe and Boyd, which contains a survey of applications of semi-definite programming.

It is also an important object when studying the question, which orbitopes are \emph{basic closed} semi-algebraic sets, i.e.~defined by finitely many simultaneous polynomial inequalities. This question has already been asked by Sanyal, Sottile and Sturmfels in their paper \cite{SS}, that initiated the study of orbitopes in their own right (in section 2 of this paper, they propose ten questions on orbitopes; our question is number 4).

Again, the notion of being basic closed relates to spectrahedra, because spectrahedra are always basic closed semi-algebraic sets.
On the other hand, for a given basic closed semi-algebraic set, Lasserre developped in \cite{LasMR2505746} a semi-definite relaxation method that was further investigated  by Helton and Nie in \cite{HeltonNieMR2515796} and \cite{HeltonNieMR2533752}. Gouveia, Parrilo and Thomas constructed semi-definite relaxations of real algebraic sets in \cite{GouMR2630035} and \cite{GouTho}. Recently, Gouveia and Netzer further investigated exactness properties of these two methods, see \cite{NetzerGouv}.

In this paper, we will restrict our attention to the special case of orbitopes of the group $\SO(2)$ of real orthogonal $2\times 2$ matrices, which is probably the simplest non-discrete case. Already in this case, the algebraic boundary is hard to describe. A first family of examples was done by Sanyal et.~al. in \cite{SS}. The authors proved that an infinite family of $\SO(2)$-orbitopes called universal $\SO(2)$-orbitopes are spectrahedra (\cite{SS}, Theorem 5.2). Their spectrahedral representation gives a determinantal representation of the irreducible polynomial defining the algebraic boundary of these convex sets.

If the $\SO(2)$-orbitope is not universal, then the algebraic boundary tends to be reducible. We will focus on the question whether or not the secant variety to the $\R$-Zariski closure of the orbit we started with is a component of the algebraic boundary. Our main result is the following statement, which deals with convex hulls of real curves in general.

\begin{Thm*}
Let $X\subset\A^{2r}$ be an irreducible curve and assume that the real points $X(\R)$ of $X$ are Zariski-dense in $X$.
Let $C$ be the convex hull of $X(\R)\subset\R^{2r}$ and suppose that the interior of $C$ is non-empty.
Then the $(r-1)$-th secant variety to $X$ is an irreducible component of the algebraic boundary of $C$ if and only if the set of all $r$-tuples of real points of $X$ that span a face of $C$ has dimension $r$.
\end{Thm*}

For two different infinite families of $\SO(2)$-orbitopes, namely $4$-dimensional $\SO(2)$-orbitopes and the family of Barvinok-Novik orbitopes, we will apply this result to prove that the appropriate secant varieties are components of their algebraic boundaries. For $4$-dimensional orbitopes, we use a complete description of their faces by Smilansky, see \cite{SmiMR815607}, to find all irreducible components of their algebraic boundary. In the case of Barvinok-Novik orbitopes, a result of the work \cite{BarvinokNovikMR2383752} will be essential.
In both cases, our results can be used to characterise, when these semi-algebraic sets are basic closed.

Namely, we will prove for Barvinok-Novik orbitopes that they are not basic closed.
For $4$-dimensional $\SO(2)$-orbitopes, we prove the following statement:
\begin{Thm*}
Let $C$ be a $4$-dimensional $\SO(2)$-orbitope. The following are equivalent:\\
(a) $C$ is linearly isomorphic to the universal $\SO(2)$-orbitope $C_2$.\\
(b) $C$ is a spectrahedron.\\
(c) $C$ is a basic closed semi-algebraic set.
\end{Thm*}

To show the basic ideas of the paper, we will informallly consider the example of the $4$-dimensional Barvinok-Novik orbitope $B_4$ (all the statements will be proved in this paper): It is by definition the convex hull of the symmetric trigonometric moment curve
\[
\{(\cos(\theta),\sin(\theta),\cos(3\theta),\sin(3\theta))\colon \theta\in[0,2\pi)\}
\]
The Zariski closure $X$ of this trigonometric curve is an algebraic curve of degree $6$ in projective $4$-space.
The orbitope $B_4$ is a simplicial and centrally symmetric convex set. By a theorem of Barvinok and Novik (\cite{BarvinokNovikMR2383752}, Theorem 1.2; or alternatively, by Smilansky's result on faces of $4$-dimensional $\SO(2)$-orbitopes in \cite{SmiMR815607}), it is locally neighbourly, i.e.~the convex hull of sufficiently close points on the trigonometric moment curve is a face of the orbitope. This allows us to explicitly compute the algebraic boundary of $B_4$; namely, it consists of two components, a quadratic hypersurface and the secant variety to the curve $X$, which is a hypersurface of degree $8$ in this case. We will show that the secant component intersects the interior of the orbitope, which proves that it cannot be basic closed. This is easier to see, if we slice the situation with the $2$-dimensional coordinate subspace $W$ spanned by the second and the last vector of the standard basis of $\R^4$. We get a $2$-dimensional semi-algebraic set (see figure \ref{fig:B4notbasic}) whose algebraic boundary has now three components, two lines and a curve of degree $3$ that goes through the origin. The explicit equations and an explanation of the line in gray can be found in Example \ref{Exm:B4}.
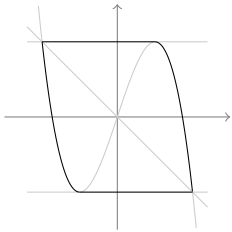
\begin{figure}[h]
\begin{center}
\begin{tikzpicture}
\draw[draw = black!50!white,->] (-1.5, 0) -- (1.5,0);
\draw[draw = black!50!white,->] (0, -1.5) -- (0, 1.5);
\draw[draw = black!20!white] (-1.2,1) -- (-1,1);
\draw[draw = black!20!white] (0.5,1) -- (1.2,1);
\draw[draw = black!20!white] (-1.2,-1) -- (-0.5,-1);
\draw[draw = black!20!white] (1,-1) -- (1.2,-1);
\draw[draw = black!20!white] (-1.2,1.2) -- (1.2,-1.2);
\draw[domain=-1.05:-1,samples=20,draw=black!20!white] plot(\x,{-4*\x*\x*\x+3*\x});
\draw[domain=-0.5:0.5,samples=200,draw=black!20!white] plot(\x,{-4*\x*\x*\x+3*\x});
\draw[domain=1:1.05,samples=20,draw=black!20!white] plot(\x,{-4*\x*\x*\x+3*\x});
\draw[domain=-1:-0.5,samples=100,draw=black] plot(\x,{-4*\x*\x*\x+3*\x});
\draw[domain=0.5:1,samples=100, draw=black] plot(\x,{-4*\x*\x*\x+3*\x});
\draw[draw=black] (-1,1) -- (0.5,1);
\draw (-0.5,-1) -- (1,-1);
\end{tikzpicture}
\end{center}
\caption{The intersection of $B_4$ with the two-dimensional coordiante subspace $W$ is the set enclosed by the black lines.}
\label{fig:B4notbasic}
\end{figure}

This article is organized as follows: In Section \ref{sec:Basics}, we will go through the basic definitions and some basic facts of the representation theory of $\SO(2)$. We will see that two orbitopes in the same representation are isomorphic, if the orbits are generically chosen.

In Section \ref{sec:Curve}, we compute the degree and singularities of the rational curve, which is the Zariski closure of an orbit under an action of $\SO(2)$. 

Section \ref{sec:SAG} is the crucial section where we apply methods from semi-algebraic geometry to prove our main result cited above. As mentioned above, it will be used in sections \ref{sec:4dim} and \ref{sec:BNOrbs} to prove that the secant variety is a component of the algebraic boundary of certain families of $\SO(2)$-orbitopes, namely the $4$-dimensional ones and the Barvinok-Novik orbitopes.
 
In Section \ref{sec:4dim}, we deal with $4$-dimensional $\SO(2)$-orbitopes, using Smilansky's characterisation of the faces of a $4$-dimensional $\SO(2)$-orbitope and the material of Section 4 to compute its algebraic boundary. Universal orbitopes are used to compute explicit equations for algebraic boundaries in some examples.

The final Section \ref{sec:BNOrbs} contains the study of the Barvinok-Novik orbitopes. Again, a secant variety is a component of the algebraic boundary and the key object in the proof of the fact that a Barvinok-Novik orbitope is not basic closed. Our proof uses a theorem of Barvinok and Novik on faces of Barvinok-Novik orbitopes from their paper \cite{BarvinokNovikMR2383752}.

\section{Setup and basic facts}\label{sec:Basics}
\begin{Def}
A \emph{representation} of $\SO(2)$ is a pair $(\rho,V)$ of a finite-dimensional real vector space $V$ and a homomorphism $\rho\colon \SO(2)\to \GL(V)$ of real algebraic groups. This means, after choosing a basis of $V$ and thereby identifying $\GL(V)$ with $\GL_n(\R)$, that $\rho$ is a group homomorphism defined by polynomials with real coefficients.\\
The dimension of the vector space $V$ is called the \emph{dimension} of the representation $(\rho, V)$.
The representation $(\rho,V)$ is called \emph{irreducible}, if $V$ has no non-trivial invariant subspace, i.e.~no subspace $W\subset V$, $\{0\}\neq W \neq V$, such that $\rho(A)w\in W$ for all $w\in W$ and $A\in \SO(2)$.\\
\end{Def}

If we have a representation $(\rho, V)$ of $\SO(2)$ then it induces an action of the group on the vector space $V$, namely $A$ acts on a vector $v$ as $\rho(A)(v)$.

\begin{Text}
We fix the following notation for representations of $\SO(2)$: For $j\in\Z$, $j\neq 0$, write
\[
\rho_j \;\colon
\left\{
\begin{array}[]{ccc}
\SO(2) &\to & \GL_2(\R) \\
A = \left(\begin{array}{cc}
\cos(\theta) & -\sin(\theta) \\
\sin(\theta) & \cos(\theta)
\end{array}\right)
&\mapsto & A^j =
\left(\begin{array}{cc}
\cos(j\theta) & -\sin(j\theta) \\
\sin(j\theta) & \cos(j\theta)
\end{array}\right) 
\end{array}\right.
\]
Denote by $\rho_0$ the trivial representation of $\SO(2)$, i.e.~the representation $(\rho,\R)$, where $\rho$ is the constant group homomorphism, i.e.~$\rho(A)=1$ in $\GL_1(\R)=\R^\times$ for all $A\in\SO(2)$.
The set $\{\rho_j\colon j\in\Z\}$ is the family of all irreducible representations of $\SO(2)$ (up to linear isomorphism commuting with the group action on $V$ via $\rho$). In particular, any representation of $\SO(2)$ that does not contain the trivial representation is even-dimensional.
\end{Text}

\begin{Rem}\label{Rem:ComplexRep}
It is often useful to switch to complex coordinates in the following sense: We identify $\SO(2)$ with the unit circle $\S^1\subset\C$ in the complex plane by sending a rotation matrix as above to $\exp(i\theta)$ and we identify $\R^2$ with $\C$ via $(x,y)\mapsto x+iy$. Then we can think of the represenation $\rho_j$ as multiplication by the exponential, i.e.~for all $z\in\C$ and $\theta\in[0,2\pi)$ we have
\[
\rho_j(\exp(i\theta))z = \exp(i\theta)^j z = \exp(ij\theta)z
\]
This is not to be confused with the complexification of the representation. The complexification is the tensor product of the representation with the field of complex numbers $\C$ over $\R$: The complexification of the group $\SO(2)$ is simply
\[
\SO(2)_\C:=\SO(2)\otimes_\R \C = \left\{
\left(\begin{array}[]{cc}
a & b\\
-b & a
\end{array}\right)\colon a,b\in \C, a^2+b^2 = 1 \right\}
=:\SO(2,\C)
\]
The complexification $\rho_j\otimes\C$ of the representation $\rho_j$ acts on $\C^2 = \R^2\otimes_\R\C$ via the same expression, i.e.~$(\rho_j\otimes\C)(A)=A^j\in\GL_2(\C)$ for all $A\in\SO(2,\C)$.\\
This representation is isomporphic to a representation of the torus $\C^\times$: Every matrix in $\SO(2,\C)$ is diagonalizable with diagonal form
\[
\left(
\begin{array}[]{cc}
\frac{1}{2} & \frac{1}{2i} \\
\frac{1}{2} & -\frac{1}{2i}
\end{array}\right)
\left(
\begin{array}[]{cc}
a & b\\
-b & a
\end{array}\right) \left(
\begin{array}[]{cc}
1 & 1 \\
i & -i
\end{array}\right)=
\left(
\begin{array}[]{cc}
a + ib & 0 \\
0 & a-ib
\end{array}\right)
\]
This change of coordinates simultaneously diagonalizes $\SO(2,\C)$. In other words, the group is conjugate in $\GL_2(\C)$ to the subgroup
\[
\left\{
\left(\begin{array}[]{cc}
a+ib & 0\\
0 & a-ib
\end{array}\right)\colon a,b\in \C, a^2+b^2 = 1\right\}
\]
of $\GL_2(\C)$ which is isomorphic to the torus $\C^\times$. The base change in $\C^2$ that corresponds to the conjugation of $\SO(2,\C)$ to this torus subgroup gives an isomorphism of the representation $\rho_j\otimes \C$ with the representation
\begin{align*}
\C^\times\times \C^2 &\to \C^2\\
(z, (x,y)) &\mapsto  (z^jx,z^{-j}y)
\end{align*}
of $\C^\times$. The real form of the torus $\C^\times$ coming from $\SO(2)$ is the unit circle, namely those $z\in \C^\times$ with $z=\frac{1}{\ol{z}}$.\\
Note that under this change of coordinates, the orbit of $(1,0)$ under the action of $\SO(2)_\C$ is mapped to the orbit of $(\frac{1}{2},\frac{1}{2})$ under the action of the torus, which is in turn isomorphic to the orbit of $(1,1)$ under the action of the torus.
\end{Rem}

\begin{Def}
Let $(\rho,V)$ be a representation of $\SO(2)$. Take $w\in V$. The convex hull of the orbit of $w$ by the action of $\SO(2)$ on $V$, i.e.~the set
\[
\conv(\rho(\SO(2))w)=\conv(\{\rho(A)w\colon A\in\SO(2)\})
\]
is called the \emph{$\SO(2)$-orbitope} of $w$ with respect to $(\rho,V)$.
\end{Def}

\begin{Rem}\label{Rem:OrbitChoice}
Fix a representation $(\rho,V)$ of $\SO(2)$.\\
(a) If there is a vector $w\in V$ such that the $\SO(2)$-orbitope of $w$ with respect to $(\rho, V)$ has non-empty interior then the representation $(\rho,V)$ must be multiplicity-free and must not contain the trivial representation as an irreducible factor.\\
(b) Any two $\SO(2)$-orbitopes with respect to $(\rho,V)$ and with non-empty interior are linearly isomorphic:
Let $j_1,\ldots,j_r\in\Z$ such that $\rho=\rho_{j_1}\oplus\ldots\oplus\rho_{j_r}$. Let $C\subset V$ be an $\SO(2)$-orbitope with non-empty interior. Say it is the convex hull of the orbit $\rho(S^1)(z_1,\ldots,z_r)\subset \C^r$ with $z_1,\ldots,z_r\in\C^\times$, i.e.~all components non-zero, which follows from the assumption that the orbitope has non-empty interior. By complex rescaling of the $j$-th irreducible component by $z_j^{-1}$ for $j=1,\ldots,r$, we get an $\R$-linear automorphism of $\C^r$ that commutes with the group action and sends the above vector to $(1,\ldots,1)\subset\C^r$. Therefore, the orbit of $(z_1,\ldots,z_r)$ is $\R$-linearly isomorphic to the orbit of $(1,\ldots,1)$.\\ 
We can also assume that the indices $j_1,\ldots,j_r$ of the irreducible components of the representation $\rho=\rho_{j_1}\oplus\ldots\oplus\rho_{j_r}$ are relatively prime, because for any $d\in\N$
\begin{align*}
&\{(\exp(idj_1\theta),\ldots,\exp(idj_r\theta))\colon\theta\in[0,2\pi)\} = \\
&\{(\exp(ij_1\theta),\ldots,\exp(ij_r\theta))\colon\theta\in[0,2\pi)\}
\end{align*}
\end{Rem}

The set of extreme points of an $\SO(2)$-orbitope is the orbit of which it is the convex hull. In their paper \cite{SS}, Sanyal, Sottile and Sturmfels proved a more general statement than the following (cf.~\cite{SS}, Proposition 2.2). In the special case of $\SO(2)$ that we are interested in, the proof is simpler and we will give it here.
\begin{Prop}\label{Prop:Extreme}
Let $(\rho,V)$ be a representation of $\SO(2)$ and let $C:=\conv(\rho(\SO(2))w)\subset V$ be an $\SO(2)$-orbitope. Then every point of the orbit of which $C$ is the convex hull is an exposed point of $C$. In particular, the orbit is the set of extreme points of $C$.
\end{Prop}

\begin{proof}
Identify $V$ with $\R^n$ by the choice of a basis of $V$.
Without loss of generality, we can assume that $C$ has non-empty interior, and therefore, that $w=(1,0,1,0,\ldots,1,0)$ (cf.~Remark \ref{Rem:OrbitChoice}).
Then the orbit $\{\rho(A)w\colon A\in\SO(2)\}$ is contained in the sphere $\{x\in\R^n\colon \|x\|=\|w\|\}$ of radius $\|w\|$ in $\R^n$. This implies that $C$ is contained in the ball of radius $\|w\|$. Since every point on the sphere is an exposed point of the ball, the claim follows.
\end{proof}

Let's see some examples.
\begin{Exm}\label{Exm:UniOrb}
Let $n\in\N$ be a natural number. Denote by $\rho\colon \SO(2)\to \GL_{2n}(\R)$ the representation $\rho_1\oplus\rho_2\oplus\ldots\oplus\rho_n$ of $\SO(2)$. The convex hull of the orbit of $(1,0,1,0,\ldots,1,0)\in\R^{2n}$ is called the \emph{universal $\SO(2)$-orbitope} of dimension $2n$. We will denote it by $C_n$. Explicitly, it is the convex hull of the trigonometric curve
\[
\{(\cos(\theta),\sin(\theta),\cos(2\theta),\sin(2\theta),\ldots,\cos(n\theta),\sin(n\theta))\in\R^{2n}\colon \theta\in[0,2\pi)\}
\]
It is called universal, because every $\SO(2)$-orbitope is the projection of a universal $\SO(2)$-orbitope.
Sanyal, Sottile and Sturmfels proved (cf.~\cite{SS}, Theorem 5.2) that the universal $\SO(2)$-orbitope $C_n$ is isomophic to the spectrahedron of positive semi-definite hermitian Toeplitz matrices of size $(n+1)\times(n+1)$ via the linear map
\[
\left\{\begin{array}{ccc}
\R^{2n} & \to & \M_{n+1}(\C) \\
(x_1,y_1,x_2,y_2,\ldots,x_n,y_n) & \mapsto &
\left(\begin{array}[]{cccc}
1 & x_1+iy_1 & \dots & x_n+iy_n \\
x_1-iy_1 & 1 & \ddots & \vdots \\
\vdots & \ddots & \ddots & x_1+iy_1 \\
x_n-iy_n & \ldots & x_1-iy_1 & 1 \\
\end{array}\right)
\end{array}\right.
\]
It follows from this theorem, that $C_n$ is an $n$-neighbourly, simplicial convex set. The maximal dimension of a face of $C_n$ is $n-1$.
\end{Exm}

\section{The Curve Associated with an $\SO(2)$-Orbitope}\label{sec:Curve}
In this paper, a variety is a variety defined over $\R$. In the language of schemes, this means, that a variety is a separated, reduced scheme of finite type over $\R$. In the classical language, it means, that an affine variety is a subset of $\C^n$ defined by real polynomials. These sets define the $\R$-Zariski topology on $\C^n$. The global sections of the sheaf of regular functions on $\A^n$ is the polynomial ring in $n$ variables over the field of real numbers $\R$. An abstract variety is a quasicompact ringed space that is locally as a ringed space isomorphic to an affine variety, as usual.\\
The real points of an affine variety $X\subset\A^n$, written as $X(\R)$, are the points in $X$ that are invariant under complex conjugation acting on $\A^n$. Analoguously for a projective variety $X\subset\P^n$.

An object of great importance for the study of an $\SO(2)$-orbitope is the Zariski closure of the orbit of which it is the convex hull. It has the following properties.
\begin{Prop}\label{Prop:CurveAss}
Let $\rho\colon \SO(2)\to \GL_n(\R)$ be a representation of $\SO(2)$. Let $O_w$ be the orbit of $w\in\R^n$, $w\neq 0$. Denote by $X$ the Zariski closure of $O_w$ in $\A^n$. Embed $\A^n\to \P^n$ via $(x_1,\ldots,x_n)\mapsto (1:x_1:\ldots:x_n)$ and denote by $\bar{X}$ the projective closure of $X$. Then $\bar{X}$ is a rational curve and the regular real points of $\bar{X}$ are exactly the orbit $O_w$, i.e.~$\bar{X}_{\rm reg}(\R)=O_w$.
\end{Prop}

\begin{proof}
Using the stereographic projection, we get a parametrisation of the unit circle $S^1$ in terms of the rational functions $R(x_0,x_1)=\frac{2x_0x_1}{x_0^2+x_1^2}$ and $I(x_0,x_1)=\frac{x_0^2-x_1^2}{x_0^2+x_1^2}$ on $\P^1$. This gives a birational map
\[
s\colon \P^1 \ratto \V(x^2+y^2-z^2)\subset\P^2=\{(x:y:z)\}
\]
Now, decompose the representation $\rho$ into irreducible factors, say $\rho=\rho_{j_1}\oplus\rho_{j_2}\oplus\ldots\oplus\rho_{j_r}$ for some $j_1,\ldots,j_r\in\N$. We can assume that $j_1,\ldots,j_r$ are relatively prime (cf.~Remark \ref{Rem:OrbitChoice}).
Since $\cos(j\theta)= f_j(\cos(\theta),\sin(\theta))$ and $\sin(j\theta)=g_j(\cos(\theta),\sin(\theta))$ are both polynomial functions of $\cos(\theta)$ and $\sin(\theta)$ for all $j\in\N$, we can define a rational map
\[
\vphi\;\colon
\left\{\begin{array}{ccc}
\V(x^2+y^2-z^2) &\ratto & \P^{2r} \\
{[} x : y :1 ] &\mapsto & [1:f_{j_1}(x,y):g_{j_1}(x,y):\ldots:f_{j_r}(x,y):g_{j_r}(x,y)]
\end{array}
\right.
\]
Its restriction to the real points of $\V(x^2+y^2-z^2)$ gives a parametrisation of the orbit $O_w$. It is injective, because the $j_i$ are relatively prime. Therefore, the restriction of the rational map $\vphi\circ s\colon \P^1\ratto \P^{2r}$ to the real points of $\P^1$ is a parametrisation of the orbit.
Since this orbit is dense in $\bar{X}$, this map $\vphi\circ s$ is a birational map onto an open subset of $\bar{X}$. This proves that $\bar{X}$ is a rational curve.\\
Since the rational functions occuring in these parametrisations have rational coefficients, we know that the image under $\vphi\circ s$ of a real point of $\P^1$ is a real point of $\bar{X}$. The converse is true on an open subset $U\subset\bar{X}$: The image of a real point of $U$ under the inverse rational map is a real point of $\P^1$. The image of $\P^1(\R)$ under $\vphi\circ s$ is closed in the euclidean topology because $\P^1(\R)$ is a compact set. The set $\bar{X}\setminus U$ is finite and therefore a real point in this set lies in $(\im(\vphi))(\R)$ or is an isolated real point and therefore singular.
\end{proof}

\begin{Def}
We call the curve $X=\cl_{Zar}(O_w)\subset\A^n$ of the preceding proposition the \emph{curve associated with} the $\SO(2)$-orbitope $\conv(O_w)$. We denote the projective closure of $X$ with respect to the embedding $\A^n\to \P^n$, $(x_1,\ldots,x_n)\mapsto (1:x_1:\ldots:x_n)$ by $\bar{X}$.
\end{Def}

\begin{Prop}\label{Prop:CurveAssDeg}
Let $\rho=\rho_{j_1}\oplus\rho_{j_2}\oplus\ldots\oplus\rho_{j_r}$ be a representation of $\SO(2)$. Let $C$ be the $\SO(2)$-orbitope of $w\in\R^{2r}$ in this representation and assume that $C$ has non-empty interior. Denote by $d=\gcd(j_1,\ldots,j_r)$ the greatest common divisor of $j_1,\ldots,j_r$ and by $j=\max\{|j_1|,\ldots,|j_r|\}$ the biggest modulus of an index. Denote by $\bar{X}$ the projective curve associated with the orbitope $C$.\\
(a) The curve $\bar{X}$ is non-singular if and only if $\frac{j}{d}-1\in \{\frac{|j_1|}{d},\ldots,\frac{|j_r|}{d}\}$.\\
(b) The degree of the curve $\bar{X}$ is $2\frac{j}{d}$.
\end{Prop}

\begin{proof}
By setting $j_i'=|\frac{j_i}{d}|$ and ordering them naturally, we assume that the $j_i$ are relatively prime and $0<j_1<j_2<\ldots<j_r=j=\frac{j}{d}$. Since $C$ has non-empty interior, we can assume after the application of a linear isomorphism of $\R^{2r}$ that $w=(1,0,1,0,\ldots,1,0)$ (cf.~Remark \ref{Rem:OrbitChoice}). We complexify the situation as explained in Remark \ref{Rem:ComplexRep} and get the following parametrisation of the complex orbit of $w$ after another change of coordinates:
\[
\C^\times\to \C^{2r}, z \mapsto (z^{j_1}, z^{-j_1}, z^{j_2}, z^{-j_2}, \ldots, z^{j_r}, z^{-j_r})
\]
(a) This map extends to a morphism $\phi\colon \{(x_0:x_1)\}=\P^1\to \P^{2r}$ which is given by the equation $\phi(1:z)=(z^{j_r}:z^{j_r+j_1}:z^{j_r-j_1}:\ldots:z^{2j_r}:1)$ on $D_+(x_0)$ and $\phi(s:1)= (s^{j_r}:s^{j_r-j_1}:s^{j_r+j_1}:\ldots:1:s^{2j_r})$ on $D_+(x_1)$. This morphism is injective: If $y,z\in\C^\times$ with $(y^{j_r},y^{j_r+j_1},y^{j_r-j_1},\ldots,y^{2j_r})=(z^{j_r},z^{j_r+j_1},z^{j_r-j_1},\ldots,z^{2j_r})$, then $(y/z)^{j_r}=1$ and therefore, $(y/z)^{j_i}=1$ for $i=1,\ldots,r$. Since the $j_i$ are relatively prime, it follows that $(y/z)\in {\rm U}(j_1)\cap\ldots\cap{\rm U}(j_r)=\{1\}$, where ${\rm U}(n)$ denotes the group of the $n$-th roots of unity.\\
The curve $\bar{X}$ is the image of this injective morphism $\phi$. If $\bar{X}$ is smooth, then $\phi$ must be an isomorphism, because the inverse rational map extends to a morphism on the non-singular curve $\bar{X}$ (\cite{FultonMR0313252}, Chapter 7, Corollary 1). In particular, if $\bar{X}$ is smooth, $\phi$ is an isomorphism of the structure sheaves and therefore, the differential is an isomorphism. This means that $\bar{X}$ is smooth if and only if the derivative of $\phi$ is non-zero at every point.\\
The derivative of $\phi$ is obviously non-zero at every point except for $(1:0)$ and $(0:1)$. It is non-zero at these points if and only if $j_r-1=j_{r-1}$, because only then $z^{j_r-j_{r-1}}=z$ and the gradient does not vanish.\\
(b) As for the degree, if we take a hyperplane $\V(a_0x_0+a_1x_1+b_1y_1+\ldots+a_rx_r+b_ry_r)\subset\P^{2r}$ and intersect it with the image of $\phi\vert_{D_+(x_0)}$, we get the identity
\[
a_0z^{j_r}+a_1z^{j_r+j_1}+b_1z^{j_r-j_1} + a_2 z^{j_r+j_2}+ b_2 z^{j_r-j_2} + \ldots + a_r z^{2j_r}+b_r = 0
\]
For a general choice of the hyperplane, this is a polynomial of degree $2j_r$ and therefore, it will have $2j_r=2\frac{j}{d}$ roots in $\C$.
\end{proof}

\begin{Cor}\label{Cor:UniversalRatNormal}
Let $\rho = \rho_1\oplus\rho_2\oplus\ldots\oplus\rho_n$ and $w=(1,0,1,0,\ldots,1,0)\in \R^{2n}$. Then the curve associated with the orbitope of $w$ in this representation, which is called the universal $\SO(2)$-orbitope of dimension $2n$ (cf.~Example \ref{Exm:UniOrb}), is a rational normal curve.
\end{Cor}

\begin{proof}
By the preceding proposition, the degree of the rational curve $\bar{X}\subset\P^{2n}$ associated with the orbitope of $w$ is $2n$ and therefore, it is a rational normal curve, cf.~\cite{HarrisiMR1182558}, Proposition 18.9.
\end{proof}

\begin{Rem}
From the proof of the above proposition, we can deduce that the real points of the projective closure of a curve $\bar{X}$ associated with an $\SO(2)$-orbitope $C$ are exactly the orbit $X(\R)=X_{\rm reg}(\R)$, of which $C$ is the convex hull: A regular point of $X$ is real if and only if it lies on the orbit (Proposition \ref{Prop:CurveAss}). If $\bar{X}$ has singular points, these are $\phi(1:0)$ and $\phi(0:1)$ in the notation of the proof of Proposition \ref{Prop:CurveAssDeg}. We change back to the original coordinates by applying
\[
\left(\begin{array}[]{cc}
1 & 1 \\
i & -i
\end{array}\right)
\]
to every irreducible factor of the representation (recall that we applied a complex change of coordinates to get an isomorphism of the complexification of our representation with a representation of the complex torus $\C^\times$, as explained in \ref{Rem:ComplexRep}). The two singular points get mapped to $(0:\ldots:0:1:i)$ and $(0:\ldots:0:1:-i)$, i.e.~a pair of complex conjugate singular points.
\end{Rem}

\section{Convex Hulls of Curves, Secant Varieties and Semi-Algebraic Geometry}\label{sec:SAG}
\begin{Def}
Let $S\subset\R^n$ be a semi-algebraic set.\\
(a) The set $S$ is called \emph{basic closed}, if there are polynomials $g_1,\ldots,g_r\in\R[X_1,\ldots,X_n]$ such that
\[
S = \{x\in\R^n\colon g_1(x)\geq 0,\ldots,g_r(x)\geq 0\}
\]
(b) The \emph{algebraic boundary} $\partial_a S$ of $S\subset\R^n$ is the Zariski closure in $\A^n$ of its boundary $\partial S$ in the euclidean topology.\\
(c) The set $S$ is called \emph{regular}, if it is contained in the closure of its interior.
\end{Def}

Note that every convex semi-algebraic set with non-empty interior is regular and its complement is also regular (possibly empty).
\begin{Lem}\label{Lem:BasicClosed}
Let $\emptyset\neq S\subset\R^n$ be a regular semi-algebraic set and suppose that its complement $\R^n\setminus S$ is also regular and non-empty.\\
(a) The algebraic boundary of $S$ is a variety of pure codimension $1$.\\
(b) If the interior of $S$ intersects the algebraic boundary of $S$ in a regular point 
then $S$ is not basic closed.
\end{Lem}

\begin{proof}
(a) By \cite{BochnakMR1659509}, Proposition 2.8.13, $\dim(\partial S)\leq n-1$. Conversely, we prove that every point in the boundary $\partial S$ of $S$ has local dimension $n-1$ in $\partial S$: Let $x\in\partial S$ be a point and take $\epsilon >0$. Then ${\rm int}(S)\cap {\rm B}(x,\epsilon)$ and ${\rm int}(\R^n\setminus S)\cap {\rm B}(x,\epsilon)$ are non-empty, because both $S$ and $\R^n\setminus S$ are regular. Applying \cite{BochnakMR1659509}, Lemma 4.5.2, yields that
\[
\dim(\partial S\cap {\rm B}(x,\epsilon)) = \dim({\rm B}(x,\epsilon)\setminus ({\rm int}(S)\cup (\R^n\setminus \ol{S}))) \geq n-1
\]
Therefore, all components of $\partial_a S={\rm cl}_{Zar}(\partial S)$ have dimension $n-1$.\\
(b) Assume that $S$ is basic closed, i.e.~there are polynomials $g_1,\ldots,g_r\in\Pol{x}{\R}$ such that $S=\{x\in\R^n\colon g_1(x)\geq 0,\ldots,g_r(x)\geq 0\}$. Let $h\in\Pol{x}{\R}$ be a polynomial defining an irreducible component of $\partial_a S$ intersecting the interior of $S$ in a regular point. Then there is an index $j\in\{1,\ldots,r\}$ and an $n\in\N$ such that $h^{2n-1}$ divides $g_j$ and $h^{2n}$ does not divide $g_j$: By canceling squares, we assume for all $k\in\{1,\ldots,r\}$ that $h$ does not divide $g_k$ or $h$ divides $g_k$ and $h^2$ does not. There is a Zariski dense subset $M\subset\V(h)_{\rm reg}(\R)$ that is contained in $\partial S$. In every element of $M$, at least one polynomial $g_k\in\{g_1,\ldots,g_r\}$ must change sign. We conclude from $M=\bigcup_{k=1}^r\V(g_k)\cap M$ that
\[
\V(h)=\ol{M}^{Zar}\subset\bigcup_{k=1}^r \V(g_k)\cap \ol{M}^{Zar}
\]
and by the irreducibility of $\V(h)$ that $\V(h)\subset \V(g_j)$ for some $j\in\{1,\ldots,r\}$.\\
But from the fact that $h$ divides $g_j$ and $h^2$ does not, we see that $g_j$ changes sign in every point of $\V(h)_{\rm reg}(\R)$ and in particular in every point of the set $\V(h)_{\rm reg}\cap \textrm{int}(S)\neq \emptyset$, which is a contradiction to $S\subset \{x\in\R^n\colon g_j(x)\geq 0\}$.
\end{proof}

\begin{Exm}
Let $g:= x^2+y^2-1\in\R[x,y]$. The union of the closed disc $\{(x,y)\in\R^2\colon g(x,y)\leq 0\}$ with the line defined by $y=0$ is basic closed, defined by the inequality $y^2g\leq 0$ and the algebraic boundary of this union has two components, namely the circle $\V(g)$ and the line $\V(y)$. The origin is a regular point of this hypersurface. This shows that the assumption on $S$ being regular in the above lemma cannot be dropped in (b).\\
For statement (a), we just have to do the same example in $\R^3$: Write $h:=x^2+y^2+z^2-1\in\R[x,y,z]$. The union of the ball $\{(x,y,z)\in\R^3\colon h(x,y,z)\leq 0\}$ with the line defined by $y=0$ and $z=0$ is basic closed, defined by the two inequalities $y^2h\leq 0$ and $z^2h\leq 0$. The algebraic boundary of this union consists of the sphere $\V(h)$ and the line $\V(y,z)$. It is a hypersurface with a lower dimensional component.
\end{Exm}

We want to characterise, when the secant variety is a component of the algebraic boundary of the convex hull of a curve.
\begin{Def}
Let $X\subset\P^n$ be an embedded quasi-projective variety. A \emph{secant $k$-plane} to $X$ is a $k$-dimensional linear space in $\P^n$ that is spanned by $k+1$ points on $X$. The \emph{$k$-th secant variety} $S_k(X)$ of $X$ is the Zariski closure of the union of all secant $k$-planes to $X$.
\end{Def}

Before we can state the theorem, we want to observe that the set of all $k$-tuples of points spanning a face is semi-algebraic:
\begin{Rem}\label{Rem:saFace}
Let $N\subset\R^n$ be a semi-algebraic set and $r\in\N$. The set $M\subset N\times\ldots\times N$ of the $r$-fold product of $N$ which contains all $r$-tuples of points whose convex hull is a face of the convex hull of $N$ is a semi-algebraic set: The set $M$ is the set of all points where a first order formula in the language of ordered fields is satisfied, namely the definition of a face, i.e.~for all $x,y\in\conv(N)$, if $\frac{1}{2}(x+y)$ is in the convex hull of the free variables $x_1,\ldots,x_r\in\R^n$, then so are $x$ and $y$.
\end{Rem}

We now come to the most important result of this section. It will be used in the following sections to show that the secant variety is an irreducible component of the algebraic boundary of certain $\SO(2)$-orbitopes. It is stated for convex hulls of not necessarily rational curves.
\begin{Thm}\label{Thm:AlgBoundOrb}
Let $X\subset\A^{2r}$ be an irreducible curve and assume that the real points $X(\R)$ of $X$ are Zariski-dense in $X$.
Let $C$ be the convex hull of $X(\R)\subset\R^{2r}$ and suppose that the interior of $C$ is non-empty. Let $M\subset X\times X\times\ldots\times X$ be the semi-algebraic subset of the $r$-fold product of $X$ defined as the set of all $r$-tuples of real points whose convex hull is a face of $C$. Then the $(r-1)$-th secant variety to $X$ is an irreducible component of the algebraic boundary of $C$ if and only if the dimension of $M$ is $r$.
\end{Thm}

\begin{proof}
The $(r-1)$-th secant variety $S_{r-1}(X)$ to $X$ is a hypersurface (cf.~\cite{LangeMR744323}), because it follows from the assumption that $C$ has non-empty interior that the curve is not contained in any hyperplane. Note that $S_{r-1}(X)$ is irreducible as the secant variety to an irreducible curve. It is contained in the algebraic boundary of $C$ if and only if the dimension of its intersection $S_{r-1}(X)\cap\partial C$ with the boundary of $C$ has codimension $1$ as a semi-algebraic set.\\
Set $M_0:=M\setminus V(\R)$ where $V\subset X\times\ldots\times X$ is the subvariety of all $r$-tuples of points on $X$ which are affinely dependent. If it is non-empty, it is a semi-algebraic set of dimension $\dim(M)$. Consider the map
\[
\Phi\colon \left\{
\begin{array}[]{ccc}
M_0\times\Delta_{r-1} & \to & \R^{2r} \\
((x_1,\ldots,x_r),(\lambda_1,\ldots,\lambda_r)) & \mapsto & \sum_{i=1}^r \lambda_i x_i
\end{array}\right.
\]
This is a semi-algebraic map and the image under $\Phi$ of $M_0\times\Delta_{r-1}$ is contained in the intersection $S_{r-1}(X)\cap\partial C$ by definition of $M_0$. We claim that $\dim(\Phi(M_0\times\Delta_{r-1}))=2r-1$ if and only if $\dim(S_{r-1}(X)\cap\partial C)=2r-1$:
If the dimension of $S_{r-1}(X)\cap\partial C$ is $2r-1$, then there exist $x\in S_{r-1}(X)\cap\partial C$ and $\epsilon>0$ such that ${\rm B}(x,\epsilon)\cap S_{r-1}(X)$ is contained in $\partial C$. Since it is also dense in $S_{r-1}(X)$, every Zariski-open subset of $S_{r-1}(X)$ intersects this set in a non-empty set, which is then open in the euclidean topology.
So $S_{r-1}(X)\cap\partial C$ contains general points of the $(r-1)$-th secant variety. Since the union of all secant $(r-1)$-planes to $X$ is a constructible set in the Zariski topology, it contains a Zariski open subset of the $(r-1)$-th secant variety. Therefore, there is a point $x\in S_{r-1}(X)\cap \partial C$ which lies on a secant $(r-1)$-plane to $X$ and an $\epsilon >0$ such that ${\rm B}(x,\epsilon)\cap S_{r-1}(X)$ is contained in the euclidean boundary of $C$ and the image of $\Phi$, i.e.~these points all lie on secant $(r-1)$-planes to $X$ spanned by real points. Therefore, the image of $\Phi$ has dimension $2r-1$.
The converse of the claimed equivalence is trivial, because $\Phi(M_0\times\Delta_{r-1})\subset S_{r-1}(X)\cap\partial C$.
From the claim, it follows that, if $S_{r-1}(X)\subset\partial_a C$, the dimension of $M_0$ is $r$ by a count of dimensions in the source of $\Phi$ and \cite{BochnakMR1659509}, Theorem 2.8.8.\\
Conversely, assume that the dimension of $M_0$ is $r$.
Denote by ${\rm Gr}(\Phi)$ the graph of the map $\Phi$ in $(M_0\times\Delta_{r-1})\times\R^{2r}$ and by $\pi_2$ the projection of this product to the second factor $\R^{2r}$. The fibre of a generic real point in $S_{r-1}(X)$ under this projection is finite, because a general point on this secant variety lies on only finitely many secant $(r-1)$-planes to $X$. This implies that the image of $\Phi$, which is the same as $\pi_2({\rm Gr}(\Phi))$, is locally homeomorphic to the graph of $\Phi$. This can be seen by a cylindrical decomposition of the semi-algebraic set ${\rm Gr}(\Phi)$ adapted to the projection $\pi_2$: Over every open cell of the decomposition of $S_{r-1}(X)(\R)$ into semi-algebraic sets, there are only graphs and no bands, so the projection $\pi_2$ is a local homeomorphism of ${\rm Gr}(\Phi_0)$ with the image of $\Phi$.
Since the graph of $\Phi$ is in turn homeomorphic to the source of $\Phi$, it follows, that the dimension of $\Phi(M_0\times\Delta_{r-1})\subset S_{r-1}(X)\cap\partial C$ is $r+r-1=2r-1$.
\end{proof}

We will mostly use this more explicit corollary to the above theorem.
\begin{Cor}\label{Cor:AlgBoundOrb}
Let $X\subset\A^{2r}$ be an irreducible curve and assume that the real points of $X$ are Zariski-dense in $X$. Set $C:=\conv(X(\R))\subset\R^{2r}$ and suppose that $C$ has non-empty interior. Then the $(r-1)$-th secant variety to $X$ is an irreducible component of the algebraic boundary of $C$ if and only if there are $r$ real points $x_1,\ldots,x_{r}\in X(\R)$ of $X$ and semi-algebraic neighbourhoods $U_j\subset X(\R)$ of $x_j$ for $j=1,\ldots,r$ such that for all $(y_1,\ldots,y_{r})\in U_1\times\ldots\times U_{r}$, the convex hull $\conv(y_1,\ldots,y_{r})$ of these points is a face of $C$.
\end{Cor}

\begin{proof}
That the $(r-1)$-th secant variety to $X$ is an irreducible component of $\partial_a C$ means that $M$ as in the above notation has dimension $r$. The euclidean topology of $X(\R)\times\ldots\times X(\R)$ is the product topology. So $M$ contains a set of the form $U_1\times\ldots\times U_r$ for open semi-algebraic sets $U_j\subset X$ if and only if it has dimension $r$.
\end{proof}

For the universal $\SO(2)$-orbitopes, the result \cite{SS}, Theorem 5.2, of Sanyal, Sottile and Sturmfels gives a complete description of the algebraic boundary (see also \ref{Exm:UniOrb}).
\begin{Exm}\label{Exm:UniBound}
The algebraic boundary of the universal $\SO(2)$-orbitope of dimension $2n$ is defined by the vanishing of the determinant
\[
\det\left(\begin{array}[]{cccc}
1 & x_1+iy_1 & \dots & x_n+iy_n \\
x_1-iy_1 & 1 & \ddots & \vdots \\
\vdots & \ddots & \ddots & x_1+iy_1 \\
x_n-iy_n & \ldots & x_1-iy_1 & 1 \\
\end{array}\right)
\]
as a polynomial in the variables $x_1,\ldots,x_n,y_1,\ldots,y_n$. It has real coefficients and is the (dehomogenization of the) equation of the $(n-1)$th secant variety to the curve $\bar{X_n}$ associated with $C_n$.\\
More generally, for $k<n$, the $k$-th secant variety to the curve $\bar{X}_n$ is defined by the $(k+2)\times(k+2)$ minors of that matrix. The union of all $k$-dimensional faces of $C_n$ is Zariski dense in the $k$-th secant variety to $\bar{X}_n$.
\end{Exm}

We take a closer look at the real points of the secant variety. We eventually show that every real points on a secant spanned by regular real points is a central point.
\begin{Def}
Let $X$ be an abstract variety. A real point $x\in X(\R)$ of $X$ is called a \emph{central point} of $X$, if it has full local dimension in the set of real points, i.e.
\[
\dim_x(X(\R))=\dim(X)
\]
\end{Def}

\begin{Rem}
Let $X\subset\P^n$ be an irreducible variety. Assume that the real points of $X$ are Zariski-dense in $X$. Let $x\in X(\R)$ be a real point. Then $x$ is a central point of $X$ if and only if there is a regular real point of $X$ in every euclidean neighbourhood of $x$.
This is a consequence of the Artin-Lang Theorem (cf.~\cite{Scheiderer2011arXiv1104.1772S}, Corollary 1.4, or \cite{KnebuschMR1029278}, Theorem II.3, Theorem III.7): If $x$ has full local dimension in the real points of $X$, then the Zariski closure of ${\rm B}(x,\epsilon)\cap X(\R)$ is $X$ for all $\epsilon > 0$ by definition of the local dimension. In particular, this neighbourhood of $x$ must contain a regular point of $X$. Conversely, if every neighbourhood of $x$ contains a regular point of $X$, then every neighbourhood has dimension $\dim(X)$ by the Artin-Lang Theorem.\\
For an alternative proof of this assertion, using the real spectrum of a ring, see \cite{BochnakMR1659509}, chapter 7, section 6.
\end{Rem}

\begin{Cor}\label{Cor:localdim}
Let $X\subset\P^n$ be an irreducible variety that is not contained in any hyperplane. Assume that the real points of $X$ are Zariski-dense in $X$. Take $x_0,\ldots,x_k\in X_{\rm reg}(\R)$ to be regular real points of $X$ that span a secant $k$-plane $\Lambda$ to $X$. Then every real point $y\in \Lambda$ is a central point of the $k$-th secant variety:
\[
\dim_y(S_k(X)(\R))=\dim(S_k(X))
\]
In particular, the union of all $k$-dimensional real projective spaces spanned by $k+1$ real points of $X$ is a Zariski-dense subset of $S_k(X)$.
\end{Cor}

\begin{proof}
The statement follows from upper semi-continuity of the local dimension, if the points $x_0,\ldots,x_k$ are general, because in that case, Terracini's Lemma (cf.~\cite{FlenMR1724388}, Proposition 4.3.2) says that the general point on the secant $k$-plane spanned by these points is a regular point of $S_k(X)$. And upper semi-continuity of the local dimension follows for example from the fact, that every closed semi-algebraic set can be locally triangulated (cf.~\cite{BochnakMR1659509}, section 9.2, Theorem 9.2.1).\\
If we take regular points $x_0,\ldots,x_k\in X_{\rm reg}(\R)$, then, since the real points of the curve $X$ are Zariski-dense in $X$, we can find for every $\epsilon>0$ a tuple $x_0',\ldots,x_k'\in X_{\rm reg}(\R)$ of general real points such that $\|x_j-x_j'\|<\epsilon$ (the reason is that ${\rm B}(x_j,\epsilon)\cap X(\R)$ is Zariski-dense in $X$ for all $\epsilon>0$ by the Artin-Lang theorem).\\
Now, if $y=\sum_{j=0}^k\lambda_j x_j$ with $\lambda_j\in\R$, $\sum_{j=0}^k\lambda_j=1$, then
\[
\|y-\sum_{j=0}^k\lambda_jx_j'\| = \|\sum_{j=0}^k\lambda_j(x_j-x_j')\|\leq\sum_{j=0}^k\lambda_j\|x_j-x_j'\|<\sum_{j=0}^k\lambda_j\epsilon = \epsilon
\]
Therefore, we can find a regular real point of $S_k(X)$ in every euclidean neighbourhood of $y$. By the preceding remark, this is equivalent to the claim.
\end{proof}

\section{Four-dimensional $\SO(2)$-orbitopes}\label{sec:4dim}
Smilansky completely charaterized the face lattice of an arbitrary $4$-dimensional $\SO(2)$-orbitope.
Let $p,q\in\N$ be relatively prime integers with $p<q$, let $\rho=\rho_p\oplus\rho_q$ be the corresponding $4$-dimensional representation of $\SO(2)$. Denote by $C_{pq}$ the convex hull of the orbit of $(1,0,1,0)$.

There is a unique pair of integers $k,\ell\in\N$ with $0\leq k<p$, $1\leq \ell <q$ such that $lp-kq=1$. Denote by $I_{pq}$ the interval
\[
I_{pq}=\left(\frac{k}{p},\frac{\ell}{q}\right)\cup \left(\frac{q-\ell}{q},\frac{p-k}{p}\right)
\]

\begin{Rem}\label{Rem:IntervalIpq}
The closure of $I_{pq}$ is the whole interval $[0,1]$ if and only if $p=1$ and $q=2$. The following case analysis shows this.
\end{Rem}

Write $z\colon [0,1]\to\R^4$, $t\mapsto(\cos(2 p\pi t),\sin(2 p\pi t), \cos(2 q\pi t),\sin(2 q\pi t))$.

\begin{Thm}[\cite{SmiMR815607}, Theorem 1]\label{Thm:4dim-faces}
The proper exposed faces of the $4$-dimensional $\SO(2)$-orbitope $C_{pq}$ are:
\begin{itemize}
\item The points $(\cos(p\theta),\sin(p\theta),\cos(q\theta),\sin(q\theta))$ ($\theta\in[0,2\pi)$) of the orbit are the $0$-dimensional faces.
\item The line segments joining $z(s)$ and $z(t)$ for $0\leq s<t<1$ with $t-s\in I_{pq}$ are $1$-dimensional faces.
\item Sets of the form $\conv(\{z(t+\frac{j}{p})\colon j=0,\ldots,p-1\})$, $0\leq t<\frac{1}{p}$ are faces. If $p\geq 3$, then these faces are two-dimensional and regular $p$-gons. If $p=2$, these are one-dimensional faces, which were not listed above. If $p=1$, these are points on the orbit listed above.
\item Analoguously for $q$: Sets of the form $\conv(\{z(t+\frac{j}{q})\colon j=0,\ldots,q-1\})$, $0\leq t<\frac{1}{q}$ are faces. If $q\geq 3$, then these faces are two-dimensional and regular $q$-gons. If $q=2$, these are one-dimensional faces, which were not listed above.
\end{itemize}
There are no three-dimensional faces.\\
If $q\geq 3$ then $C_{pq}$ has non-exposed faces, namely the edges of the $q$-gons listed above. If $p$ is also greater than $2$, then the edges of the $p$-gons are also non-exposed faces of $C_{pq}$.
\end{Thm}

\begin{Rem}
This theorem has interesting immediate consequences:\\
(a) A $4$-dimensional $\SO(2)$-orbitope $C_{pq}$ is simplicial (i.e.~all faces are simplices) if and only if $(p,q)=(1,2)$ or $(p,q)=(1,3)$.\\
(b) A four-dimensional $\SO(2)$-orbitope has non-exposed faces if and only if $(p,q)\neq (1,2)$. Combining this with the theorem of Sanyal, Sottile and Sturmfels about universal orbitopes (cf.~\cite{SS}, Theorem 5.2) and the fact that every face of a spectrahedron is exposed (cf.~\cite{GoldmanMR1342934}, Corollary 1), it is immediate, that a $4$-dimensional $\SO(2)$-orbitope is a spectrahedron if and only if it is universal. An even stronger statement is true (cf.~Corollary \ref{Cor:4dim}).
\end{Rem}

We investigate the algebraic boundary of $4$-dimensional $\SO(2)$-orbitopes.
\begin{Thm}\label{Thm:AlgBound4}
Let $p$ and $q$ be relatively prime integers, $q>p$. Choose the coordinates $\R^4\subset\A^4 = \{(w,x,y,z)\}$ and denote by $X_{pq}$ the curve associated with $C_{pq}$. The algebraic boundary of $C_{pq}$ is
\[
\partial_a C_{pq} = \left\{
\begin{array}[]{cc}
S_1(X_{pq}) & \text{ if } p=1, q=2 \\
S_1(X_{pq})\cup \V(y^2+z^2-1) & \text{ if } p\in\{1,2\}, q\geq 3 \\
S_1(X_{pq})\cup \V(w^2+x^2-1) \cup \V(y^2+z^2-1) & \text{ if } p\geq 3
\end{array}\right.
\]
\end{Thm}

\begin{proof}
The fact that the secant variety to the curve $X_{pq}$ associated with $C_{pq}$ is a component of the algebraic boundary of $C_{pq}$ follows from Theorem \ref{Thm:AlgBoundOrb} and the list of $1$-dimensional faces of $C_{pq}$ because there is always a $2$-dimensional family of edges.\\
The case of the universal $4$-dimensional orbitope, i.e.~$p=1$, $q=2$, follows from \cite{SS}, Theorem 5.2 (cf.~Example \ref{Exm:UniBound}).\\
Next, consider the case $p=1$ or $p=2$ and $q\geq 3$. Then the boundary of $C_{pq}$ consists of a $2$-dimensional family of edges and a $1$-dimensional family of regular $q$-gons.
The union of the $q$-gons is a semi-algebraic set of dimension $3$: Consider the semi-algebraic map
\[
\left\{
\begin{array}[]{c}
(0,1)\times{\rm relint}(\Delta_{2})\to \R^4\\
(t,(\lambda_0,\lambda_1,\lambda_2))\mapsto \lambda_0z(t)+\lambda_1z(t+\frac{1}{q})+\lambda_2z(t+\frac{2}{q})
\end{array}\right.
\]
which is injective, because $3$ vertices of a regular $q$-gon are affinely independent and the relative interiors of the $q$-gons in the boundary of $C_{pq}$ are disjoint. By \cite{BochnakMR1659509}, Theorem 2.8.8, it follows that the image has dimension $3$. To calculate the Zariski closure of this set, note that the last two components of the vectors $z(t)$, $z(t+\frac{1}{q})$ and $z(t+\frac{2}{q})$ are equal and therefore, the same is true for every element in the convex hull of these $3$ points. This implies, that the image is contained in the hypersurface $\V(y^2+z^2-1)$, which is irreducible. Therefore, the Zarsiki closure of the image is this hypersurface.
This shows $S_1(X_{pq})\cup\V(y^2+z^2-1)\subset\partial_a C_{pq}$ and since every face of $C_{pq}$ is contained in this variety, there are no further components in this case.\\
The case $p\geq 3$ is completely analoguous to the last case. The new component $\V(w^2+x^2-1)$ is the Zariski closure of the regular $p$-gons that lie in the boundary of $C_{pq}$.
\end{proof}

\begin{Cor}\label{Cor:4dim}
Let $C$ be a $4$-dimensional $\SO(2)$-orbitope. The following are equivalent:\\
(a) $C$ is linearly isomorphic to the universal $\SO(2)$-orbitope $C_2$.\\
(b) $C$ is a spectrahedron.\\
(c) $C$ is a basic closed semi-algebraic set.
\end{Cor}

\begin{proof}
The implication from (a) to (b) is \cite{SS}, Theorem 5.2, (b) to (c) is linear algebra: A spectrahedron $\{(x_1,\ldots,x_n)\in\R^n\colon A_0+x_1A_1+\ldots+x_nA_n \geq 0 \}$ can be defined in terms of polynomial inequalities by simultaneous sign conditions on the minors of the matrix inequality. We prove the implication from (c) to (a) by contraposition:\\
Let $C$ be a $4$-dimensional $\SO(2)$ orbitope which is not linearly isomorphic to the universal orbitope. Then the algebraic boundary consists of at least two components, one of which is the secant variety to the curve $X$ associated with the orbitope $C$. The list of all faces shows that there is a line segment joining two points $X_{\rm reg}(\R)$ of the orbit associated with $C$ that intersects the interior of $C$ (cf.~Remark \ref{Rem:IntervalIpq}). This point has full local dimension in the real points of the secant variety $S_1(X)(\R)$ to $X$ by Corollary \ref{Cor:localdim}. By Lemma \ref{Lem:BasicClosed} we conclude, that $C$ is not basic closed.
\end{proof}

\begin{Exm}
(a) We explicitly compute the algebraic boundary of the $4$-dimensional Barvinok-Novik orbitope $B_4=C_{13}$ - we will introduce the family of Barvinok-Novik orbitopes in section \ref{sec:BNOrbs}. This means that we have to compute the equation of the secant variety to the curve $X_{13}$ associated with $C_{13}$. We will use the ideal defining the secant variety to the curve associated with the universal $\SO(2)$-orbitope $C_3$, which is given by the $3\times 3$ minors of the linear matrix inequality defining $C_3$, cf.~Example \ref{Exm:UniBound}.\\
The union of all lines joinig two general real points of $X_{13}$ is a Zariski-dense subset of the secant variety $S_1(\bar{X}_{13})$ because the real points of $X_{13}$ are by definition Zariski-dense in $X_{13}$ (cf.~Corollary \ref{Cor:localdim}). The projection from $\R^6$ to $\R^4$ that projects $C_3$ onto $C_{13}$ gives a bijection, if restricted to the union of all lines joining two real points of the curve $X_3$ associated with $C_3$ because it is a bijection, if restricted to the orbit $X_{3_{\rm reg}}(\R)$. Therefore, the secant variety $S_1(X_{13})$ is the image of the secant variety $S_1(X_3)$ under this projection. This leads to an elimination problem. The author solved it using the computer algebra system \texttt{Macaulay2} \cite{Macaulay2}. In the coordinates $\A^4=\{w,x,y,z\}$, the equation of the secant variety is the following polynomial $f$ of degree $8$ and $47$ terms:
{\small
\begin{align*}
f = &-36w^4x^2y^2+24w^2x^4y^2-4x^6y^2+24w^5xyz-80w^3x^3yz+24wx^5yz-4w^6z^2\\
&+24w^4x^2z^2-36w^2x^4z^2+4w^6+12w^4x^2+12w^2x^4+4x^6-12w^5y+24w^3x^2y\\
&+36wx^4y+12w^4y^2+24w^2x^2y^2+12x^4y^2-4w^3y^3+12wx^2y^3-36w^4xz\\
&-24w^2x^3z+12x^5z-12w^2xy^2z+4x^3y^2z+12w^4z^2+24w^2x^2z^2+12x^4z^2\\
&-4w^3yz^2+12wx^2yz^2-12w^2xz^3+4x^3z^3-3w^4-6w^2x^2-3x^4+8w^3y-24wx^2y\\
&-6w^2y^2-6x^2y^2+y^4+24w^2xz-8x^3z-6w^2z^2-6x^2z^2+2y^2z^2+z^4
\end{align*}
}
(b) We compute, as above, the equation of secant variety to the curve $X_{14}$ associated with the $\SO(2)$-orbitope $C_{14}$:
We eliminate the $4$ variables of the projection from the universal $\SO(2)$-orbitope $C_4$ onto $C_{14}$ from the ideal of the $3\times 3$ minors of the linear matrix inequality describing $C_4$, i.e.~project the secant variety to the curve associated with $C_4$ to the space in which $C_{14}$ lives. We obtain an equation of degree $15$ and $281$ terms.
{\small
\begin{align*}
g  = & \;\, 8w^{12}y^3  -  144w^{10}x^2y^3  +  888w^8x^4y^3 - 2016w^6x^6y^3 + 888w^4x^8y^3 - 144w^2x^{10}y^3 \\
&+ 8x^{12}y^3 + 96w^{11}xy^2z - 1248w^9x^3y^2z + 4800w^7x^5y^2z - 4800w^5x^7y^2z \\
&+ 1248w^3x^9y^2z - 96wx^{11}y^2z + 384w^{10}x^2yz^2 - 3072w^8x^4yz^2 + 5376w^6x^6yz^2 \\
&- 3072w^4x^8yz^2 + 384w^2x^{10}yz^2 + 512w^9x^3z^3 - 1536w^7x^5z^3 + 1536w^5x^7z^3 \\
&- 512w^3x^9z^3 + 8w^{12}y^2 - 80w^{10}x^2y^2 + 120w^8x^4y^2 + 416w^6x^6y^2 + 120w^4x^8y^2 \\
&- 80w^2x^{10}y^2 + 8x^{12}y^2 + 64w^{11}xyz - 320w^9x^3yz - 384w^7x^5yz + 384w^5x^7yz \\
&+ 320w^3x^9yz - 64wx^{11}yz + 128w^{10}x^2z^2 - 256w^6x^6z^2 + 128w^2x^{10}z^2 - 8w^{12}y \\
&+ 16w^{10}x^2y + 136w^8x^4y + 224w^6x^6y + 136w^4x^8y + 16w^2x^{10}y - 8x^{12}y \\
&- 72w^{10}y^3 + 216w^8x^2y^3 + 1008w^6x^4y^3 + 1008w^4x^6y^3 + 216w^2x^8y^3 - 72x^{10}y^3 \\
&- 32w^{11}xz - 96w^9x^3z -64w^7x^5z + 64w^5x^7z + 96w^3x^9z + 32wx^{11}z - 288w^9xy^2z \\
&- 576w^7x^3y^2z + 576w^3x^7y^2z + 288wx^9y^2z - 72w^{10}yz^2 + 216w^8x^2yz^2 \\
&+ 1008w^6x^4yz^2 + 1008w^4x^6yz^2 + 216w^2x^8yz^2 - 72x^{10}yz^2 - 288w^9xz^3 \\
&- 576w^7x^3z^3 + 576w^3x^7z^3 + 288wx^9z^3 - 8w^{12} - 48w^{10}x^2 - 120w^8x^4 \\
&- 160w^6x^6 - 120w^4x^8 - 48w^2x^{10} - 8x^{12} - 8w^{10}y^2 + 16w^2y^2 + \\
\ldots &  + 16x^2y^2 + 2y^4 + 16w^2z^2 + 16x^2z^2 + 4y^2z^2 + 2z^4 - y^2 - z^2
\end{align*}
}
\end{Exm}

\section{Barvinok-Novik orbitopes}\label{sec:BNOrbs}
\begin{Def}
For any odd integer $n\in\N$, we consider the direct sum of representations $\rho=\rho_1\oplus \rho_3\oplus\ldots\oplus\rho_n$ of $\SO(2)$ indexed by all odd integers from $1$ to $n$. The $(n+1)$-dimensional \emph{Barvinok-Novik orbitope} is the convex hull of the orbit of $(1,0,1,0,\ldots,1,0)\in (\R^2)^{\frac{n+1}{2}}$ under the representation $\rho$. Explicitly, it is the convex hull of the symmetric trigonometric moment curve
\[
\{(\cos(\theta),\sin(\theta),\cos(3\theta),\sin(3\theta),\ldots,\cos(n\theta),\sin(n\theta))\colon \theta\in[0,2\pi]\}
\]
We parametrise this orbit by the map
\[
\textrm{SM}_{n+1}\;\colon
\left\{
\begin{array}[]{ccc}
\textrm{S}^1 &\to& \R^{n+1},\\
\exp(i\theta) &\mapsto& (\cos(\theta),\sin(\theta),\cos(3\theta),\sin(3\theta),\ldots,\cos(n\theta),\sin(n\theta))
\end{array}\right.
\]
\end{Def}

\begin{Rem}
Barvinok-Novik orbitopes are centrally symmetric, i.e.~$-B_{n+1}=B_{n+1}$, because $\textrm{SM}_{n+1}(\exp(it+i\pi))=-\textrm{SM}_{n+1}(\exp(it))$ for all $t\in\R$.
\end{Rem}

\begin{Prop}
Every Barvinok-Novik orbitope is a simplicial compact convex set.
\end{Prop}

\begin{proof}
Let $n\geq 3$ be an odd integer. The maximal dimension of a face of $B_{n+1}$ is $n-1$ because the boundary of $B_{n+1}$ has dimension $n$ and a face always comes with an orbit of faces.
We will prove that any $n+1$ points on the orbit
\[
\{(\cos(\theta),\sin(\theta),\cos(3\theta),\sin(3\theta),\ldots,\cos(n\theta),\sin(n\theta))\colon\theta\in[0,2\pi)\}
\]
are $\R$-affinely linearly independent. We switch to complex coordinates: Take pairwise distinct points $\theta_0,\theta_1,\ldots,\theta_n\in[0,2\pi)$. Set $z_j=\exp(i\theta_j)$. The corresponding points $(z_j,z_j^3,\ldots,z_j^n)$ on the orbit are $\R$-affinely linearly independent if and only if we can conclude from $a_j\in\R$ ($j=0,\ldots,n$) and the equations
\[
\begin{array}[]{ccc}
\sum_{j=0}^n a_j & = & 0 \\
\sum_{j=0}^n a_j \left(
\begin{array}[]{c}
z_j \\ z_j^3 \\ \vdots \\ z_j^n
\end{array}\right) & = & 0
\end{array}
\]
that all the coefficients $a_j$ are zero. This is true if the $(n+2)\times(n+1)$-matrix
\[
\left(\begin{array}[]{cccc}
\ol{z_0}^n & \ol{z_1}^n & \ldots & \ol{z_n}^n \\
\vdots & \vdots & & \vdots \\
\ol{z_0}^3 & \ol{z_1}^3 & \ldots & \ol{z_n}^3 \\
\ol{z_0} & \ol{z_1} & \ldots & \ol{z_n} \\
1 & 1 & \ldots & 1\\
z_0 & z_1 & \ldots & z_n \\
z_0^3 & z_1^3 & \ldots & z_n^3 \\
\vdots & \vdots & & \vdots \\
z_0^n & z_1^n & \ldots & z_n^n
\end{array}\right)
\]
has full rank $n+1$ over $\C$. By using Vandermonde's rule, we prove that the determinant of the matrix obtained from the one above by deleting the row of ones does not vanish. To see this, we first rescale the $j$-th column of the new $(n+1)\times(n+1)$ matrix, i.e.~the column with $z_{j-1}$, by $z_{j-1}^{-1}$. Using the identity $z_j^{-1}\ol{z_j}=z_j^{-2}$, we get the matrix
\[
\left( \begin{array}[]{cccc}
z_0^{-2}\ol{z_0}^{n-1} & z_1^{-2}\ol{z_1}^{n-1} & \ldots & z_n^{-2}\ol{z_n}^{n-1} \\
\vdots & \vdots & & \vdots \\
z_0^{-2}\ol{z_0}^2 & z_1^{-2}\ol{z_1}^2 & \ldots & z_1^{-2}\ol{z_n}^2 \\
z_0^{-2} & z_1^{-2} & \ldots & z_n^{-2} \\
1 & 1 & \ldots & 1\\
z_0^2 & z_1^2 & \ldots & z_n^2 \\
\vdots & \vdots & & \vdots \\
z_0^{n-1} & z_1^{n-1} & \ldots & z_n^{n-1}
\end{array}\right) 
\]
This gives a Vandermonde matrix, if we rescale the $j$-th column by $z_{j-1}^{n+1}$ and substitute $y_j:=z_j^2$. We conclude that the determinant of this matrix does not vanish. This proves the claim.
\end{proof}

\begin{Rem}
The Barvinok-Novik orbitope $B_{n+1}$ of dimension $n+1$ has faces of the highest possible dimension $n-1$. Namely, Barvinok and Novik prove in \cite{BarvinokNovikMR2383752}, section 3.4, that for all $\theta\in[0,2\pi)$ the set
\[
\conv(\{SM_{n+1}(\exp(i\theta + \frac{2\pi i}{n}j))\colon j=0,\ldots,n-1\})
\]
is a simplicial face of $B_{n+1}$.\\ 
An easy computation shows that these faces are exposed by hyperplanes of the form $\{x\in\R^{n+1}\colon x^t w = 1 \}$ for a unit vector $w=(0,\ldots,0,w_n,w_{n+1})\in\R^{n+1}$ whose first $n-1$ entries are zero.
\end{Rem}

We have already seen that the $4$-dimensional Barvinok-Novik orbitope $B_4$ is not basic closed (cf.~Corollary \ref{Cor:4dim}). We proved this by calculating the algebraic boundary of $B_4$. By the description of its faces, we saw that the algebraic boundary of this convex set intersects the interior in regular points. The component on which these points lie is the secant variety to the curve associated with $B_4$. We will now generalise this to higher dimensions by examining higher secant varieties.

We will need the following result due to Barvinok and Novik on the existence of faces of $B_{n+1}$ of appropriate dimension. It says, informally speaking, that the Barvinok-Novik orbitope is locally neighbourly.
\begin{Thm}[\cite{BarvinokNovikMR2383752}, Theorem 1.2]\label{Thm:BNSimp}
Let $n\geq 3$ be an odd integer. For every integer $j \leq \frac{n-1}{2}$ there is some constant $\phi_j$ such that any $j+1$ points on $S^1$ lying on an arc of length smaller than $\phi_j$ define a simplicial face of $B_{n+1}$. This face is the convex hull of the images of these $j+1$ points under the map $\textrm{SM}_{n+1}$.
\end{Thm}

\begin{Thm}
Let $n\in\N$ be an odd integer greater than $2$. Denote by $X_{n+1}$ the curve associated with the $(n+1)$-dimensional Barvinok-Novik orbitope $B_{n+1}$. The $\frac{n-1}{2}$-th secant variety to $\bar{X}_{n+1}$ is an irreducible component of the algebraic boundary of $B_{n+1}$.
\end{Thm}

\begin{proof}
Set $k:=\frac{n-1}{2}$. Firstly, the origin is an interior point of the Barvinok-Novik orbitope $B_{n+1}$ because it is an interior point of all universal $\SO(2)$-orbitopes (cf.~\cite{SS}, Theorem 5.2) and $B_{n+1}$ is a linear projection of $C_{n}$. Therefore, $X_{n+1}$ is not contained in any hyperplane.
So by \cite{LangeMR744323}, the dimension of $S_k(X_{n+1})$ equals $2k+1=2\frac{n-1}{2}+1 = n$. Because it is the secant variety to an irreducible curve, it is irreducible.\\
To see that it is a component of the algebraic boundary of $B_{n+1}$, use Corollary \ref{Cor:AlgBoundOrb} with points ${\rm SM}_{n+1}(\exp(it_0)),\ldots,{\rm SM}_{n+1}(\exp(it_k))\in X(\R)$, where $t_0,\ldots,t_k\in[0,2\pi)$ are chosen such that the points $\exp(it_0),\ldots,\exp(it_k)$ lie on an arc on $S^1\subset\C$ of length smaller than the constant $\varphi_k>0$ of the above Theorem \ref{Thm:BNSimp}. Choose sufficiently small semi-algebraic neighbourhoods $U_j={\rm B}(x_j,\epsilon_j)\cap X(\R)$ of $x_j$ satisfying the hypothesis of Theorem \ref{Cor:AlgBoundOrb}. The existence of a sufficiently small $\epsilon_j>0$ is guaranteed by Theorem \ref{Thm:BNSimp}.
\end{proof}

\begin{Cor}
No Barvinok-Novik orbitope is a basic closed semi-algebraic set.
\end{Cor}

\begin{proof}
The $\frac{n-1}{2}=:k$-th secant variety to the curve $X_{n+1}$ associated with $B_{n+1}$ is a component of the algebraic boundary. The origin lies on this component because it lies on the line joining $\textrm{SM}_{n+1}(1)$ and $\textrm{SM}_{n+1}(-1)$. It is a central point of $S_k(X_{n+1})$ by Corollary \ref{Cor:localdim}, i.e.~in every euclidean neighbourhood of the origin there is a regular point of $S_k(X_{n+1})$. By Lemma \ref{Lem:BasicClosed}, this implies that the Barvinok-Novik orbitope is not basic closed.
\end{proof}

In the special case of $B_4$, we look into this argument more concretely by considering a fortunately chosen slice of the convex set.
\begin{Exm}\label{Exm:B4}
We intersect $B_4$ with the subspace $W:=\{(0,x,0,z)\in\R^4\colon x,z\in\R\}$.
The polynomials defining the irreducible components of $\partial_a B_4$ restricted to this subspace factor $0^2+z^2-1=(z+1)(z-1)$ and $f(0,x,0,z)=(x+z)^3(4x^3-3x+z)$ (cf.~figure \ref{fig:B4notbasic}). The polynomial $4x^3-3x+z$ is part of the algebraic boundary of the convex and semi-algebraic set $W\cap B_4$ but the origin is an interior point of $W\cap B_4$ and a regular point of the hypersurface $\V(4x^3-3x+z)$. Using Lemma \ref{Lem:BasicClosed}, we can conclude from this that $B_4$ is not basic closed.
\end{Exm}

{\small\linespread{1}

}
\end{document}